\documentclass[12pt,reqno]{amsart}

\usepackage{multirow}
\usepackage{hhline}

\usepackage{mathtools}
\usepackage{times}
\usepackage[T1]{fontenc}
\usepackage{mathrsfs}
\usepackage{latexsym}
\usepackage[dvips]{graphics}
\usepackage[titletoc, title]{appendix}
\usepackage{amsmath,amsfonts,amsthm,amssymb,amscd}
\usepackage[dvipsnames]{xcolor}
\usepackage{hyperref}
\usepackage{amsmath}
\usepackage[utf8]{inputenc}

\usepackage{color}
\usepackage{breakurl}

\usepackage{comment}
\newcommand{\bburl}[1]{\textcolor{blue}{\url{#1}}}

\newtheorem{thm}{Theorem}[section]

\newtheorem{cor}[thm]{Corollary}

\newtheorem{lem}[thm]{Lemma}
\newtheorem{prop}[thm]{Proposition}

\newtheorem{que}[thm]{Question}

\usepackage[utf8]{inputenc}

\DeclareFixedFont{\ttb}{T1}{txtt}{bx}{n}{12} 
\DeclareFixedFont{\ttm}{T1}{txtt}{m}{n}{12}  

\usepackage{color}
\definecolor{deepblue}{rgb}{0,0,0.5}
\definecolor{deepred}{rgb}{0.6,0,0}
\definecolor{deepgreen}{rgb}{0,0.5,0}

\usepackage{listings}

\newcommand\pythonstyle{\lstset{
language=Python,
basicstyle=\ttm,
morekeywords={self},              
keywordstyle=\ttb\color{deepblue},
emph={MyClass,__init__},          
emphstyle=\ttb\color{deepred},    
stringstyle=\color{deepgreen},
frame=tb,                         
showstringspaces=false
}}

\lstnewenvironment{python}[1][]
{
\pythonstyle
\lstset{#1}
}
{}

\newcommand\pythoninline[1]{{\pythonstyle\lstinline!#1!}}

\numberwithin{equation}{section}

\DeclareFontFamily{U}{mathx}{}
\DeclareFontShape{U}{mathx}{m}{n}{<-> mathx10}{}
\DeclareSymbolFont{mathx}{U}{mathx}{m}{n}
\DeclareMathAccent{\widehat}{0}{mathx}{"70}
\DeclareMathAccent{\widecheck}{0}{mathx}{"71}

\begin{document}

\title{Fixed-Term Decompositions Using Even-Indexed Fibonacci Numbers}
\author{H\`ung Vi\d{\^e}t Chu}
\address{Department of Mathematics, Texas A\&M University, College Station, TX 77843, USA}
  \email{hungchu1@tamu.edu}

\author{Aney Manish Kanji}
\address{Department of Mathematics, Texas A\&M University, College Station, TX 77843, USA}
\email{aneykanji\_tamu@tamu.edu}

\author{Zachary Louis Vasseur}
\address{Department of Mathematics, Texas A\&M University, College Station, TX 77843, USA}
\email{zachary.l.v@tamu.edu}
\thanks{The work was partially supported by the College of Arts \& Sciences at Texas A\&M University. A. M. Kanji and Z. L. Vasseur are undergraduates at Texas A\&M University.}



\maketitle

\begin{abstract}
As a variant of Zeckendorf's theorem, Chung and Graham proved that every positive integer can be uniquely decomposed into a sum of even-indexed Fibonacci numbers, whose coefficients are either $0, 1$, or $2$ so that between two coefficients $2$, there must be a coefficient $0$. This paper characterizes all positive integers that do not have $F_{2k}$ ($k\ge 1$) in their decompositions. This continues the work of Kimberling, Carlitz et al., Dekking, and Griffiths, to name a few, who studied such a characterization for Zeckendorf decomposition. 
\end{abstract}

\section{Introduction}
We define the Fibonacci sequence $(F_n)_{n=1}^\infty$ as $F_1 = F_2 = 1$ and $F_{n+1} = F_{n} + F_{n-1}$ for $n\ge 2$. 
Zeckendorf's theorem \cite{Ze} states that every positive integer can be uniquely written as a sum of nonadjacent Fibonacci numbers from $(F_n)_{n=2}^\infty$. The sum is called the \textit{Zeckendorf decomposition} of a positive integer. Note that we start from $F_2$ since otherwise, $F_1 = F_2 = 1$ ruins uniqueness. Zeckendorf-type decompositions have been extensively studied in the literature: to name a few, see \cite{CFHMN, Cha, CLM, Da1, Da2, DDKMMV, DFFHMPP, HW, MMMS, MMMMS} for various generalizations to other sequences, \cite{Ba, De1, De2, DS, Gr1, Ho, Sh} for digits in the decomposition, and \cite{BEFM1, BEFM2, LLMMSXZ, MSY} for Zeckendorf games. 

A beautiful Zeckendorf-type decomposition that uses even-indexed Fibonacci numbers only is due to Chung and Graham \cite{CG}:

\begin{thm}\cite[Lemma 1]{CG}\label{CGlemma}
    Every positive integer $n$ can be uniquely represented as a sum
    $n = \sum_{i\ge 1} c_i F_{2i}$, where $c_i$'s are in $\{0, 1, 2\}$ so that if $c_i = c_j = 2$ with $i < j$, then for some $k$, $i < k < j$, we have $c_k = 0$. 
\end{thm}

We call the decomposition in Theorem \ref{CGlemma} the \textit{Chung-Graham decomposition} of an integer. 
This paper answers the following question.
\begin{que}\normalfont\label{mainquestion}
Given an even-indexed Fibonacci number $F_{2N}$, $N\ge 1$, what are the positive integers whose Chung-Graham decomposition contains neither $F_{2N}$ nor $2F_{2N}$?    
\end{que}

The analog of Question \ref{mainquestion} is well-known for Zeckendorf decomposition. Kimberling \cite{Kim} studied numbers without $1$ in their decomposition. Another pioneering paper is due to Carlitz et al.\ \cite{CHS} who described the set $Z(N)$ of all positive integers having the summand $F_N$ in their Zeckendorf decomposition, depending on the parity of $N$. Later, Griffiths \cite{Gr2} gave 
$$Z(N)\ =\ \left\{F_N\left\lfloor \frac{n+\phi^2}{\phi}\right\rfloor + nF_{N+1} + j\,:\, 0\le j\le F_{N-1} - 1, n\ge 0\right\}.$$ 
Dekking \cite{De2} characterized all integers that share the same \textit{initial} Zeckendorf decomposition, using the so-called compound Wythoff sequences and generalized Beatty sequences. 

It is worth mentioning that Griffiths' analysis \cite{Gr2} can also be used to determine all positive integers having $\{F_N: N\in A\}$ for some certain sets $A$  in their Zeckendorf decomposition. The idea is to analyze consecutive rows of the table of all numbers having $F_N$ as the minimum summand in their Zeckendorf decomposition and employed properties of the golden string, which we shall discuss in Section \ref{gold}. Recently, Chu \cite{Ch} generalized the golden string to study a generalized Zeckendorf decomposition.

In the present paper, we answer Question \ref{mainquestion} using the same method as in \cite{Ch, Gr2} while dealing with a considerably more involved table due to the appearance of the coefficient $2$ in the Chung-Graham decomposition. In the process, we need to utilize more properties of the golden string (see Propositions \ref{kp1} and \ref{kp2}, for example). We state our main result. 

\begin{thm}\label{mtheo}
For $N\ge 1$, the set of all positive integers that do not have $F_{2N}$ nor $2F_{2N}$ in their Chung-Graham decomposition is given by 
\begin{align*} B_{2N}&\ :=\ [1, F_{2N}-1]\cup \\
&\bigcup_{k=N+1}^\infty \left\{j + F_{2k}, j+(n+2)F_{2k} + \left\lfloor \frac{n+1}{\phi}\right\rfloor F_{2k-1}\,:\, 0\le j\le F_{2N}-1, n\ge 0\right\},\end{align*}
where $\phi = (\sqrt{5}+1)/2$.
\end{thm}

For example, we list integers at most $30$ that belong to the following sets
\begin{align*}
    B_2&\ =\ \{3,6,8,11,14,16,19,21,24,27,29,\ldots\},\\
    B_4&\ =\ \{1,2,8,9,10, 16, 17, 18, 21, 22, 23, 29, 30\},\\
    B_6&\ =\ \{1,2,3,4,5,6,7,21, 22, 23, 24, 25, 26, 27, 28\}.
\end{align*}

To facilitate our writing, we introduce some notation that distinguish the two coefficients $1$ and $2$. Given $n\in \mathbb{N}$, let $\mathcal{CG}(n)$ denote the set of all Fibonacci numbers in the Chung-Graham decomposition of $n$. Let $\mathcal{CG}_1(n)$ be the set of all numbers in $\mathcal{CG}(n)$ that have coefficient $1$ in the Chung-Graham decomposition of $n$, and let $\mathcal{CG}_2(n) := \mathcal{CG}(n)\backslash \mathcal{CG}_1(n)$ be the set of all numbers in $\mathcal{CG}(n)$ that have coefficient $2$. For example,
\begin{align*}
    \mathcal{CG}(2F_{2} + F_4 + 2F_8 + F_{14})&\ =\ \{F_2, F_4, F_8, F_{14}\},\\
    \mathcal{CG}_1(2F_{2} + F_4 + 2F_8 + F_{14})&\ =\ \{F_4, F_{14}\}, \mbox{ and }\\
    \mathcal{CG}_2(2F_{2} + F_4 + 2F_8 + F_{14})&\ =\ \{F_2, F_8\}.
\end{align*}
For example, $\max \mathcal{CG}(n) = \max\mathcal{CG}_1(n) = F_{14}$ means that the largest Fibonacci number that appears in the Chung-Graham decomposition of $n$ is $F_{14}$ whose coefficient is $1$.  

The paper is structured as follows: in Section \ref{gold}, we define the golden string $\mathcal{S}$ and collect several properties that will be used in due course; Section \ref{tableprop} investigates the ordered list of integers having $F_{2k}$ as the smallest Fibonacci number in their Chung-Graham decomposition; finally, Section \ref{maintheoproof} gathers some auxiliary results before proving Theorem \ref{mtheo}.

\section{The golden string}\label{gold}
For two finite strings of symbols $X$ and $Y$, we write $X:Y$ to mean the concatenation of $X$ and $Y$. The golden string, denoted by $\mathcal{S}$, is an infinite string consisting of the letters $A$ and $B$, built recursively as follows: $S_1 = B$, $S_2 = BA$, and $S_k = S_{k-1}:S_{k-2}$ for $k\ge 3$. For example, 
\begin{align*}
    S_3 &\ =\ S_2:S_1\ =\ BAB,\\
    S_4 &\ =\ S_3:S_2\ =\ BABBA,\\
    S_5 &\ =\ S_4:S_3\ =\ BABBABAB.
\end{align*}
The first few letters of $\mathcal{S}$ are
$$BABBABABBABBABABBABAB\ldots.$$
We record several properties of $\mathcal{S}$ and $(S_n)_{n=1}^\infty$:
\begin{enumerate}
\item[a)] The length of $S_n$, denoted by $|S_n|$, is equal to $F_{n+1}$.
\item[b)] The substring $S_n$ gives the first $F_{n+1}$ letters of $\mathcal{S}$. 
\item[c)] For each $n\ge 2$, the substring consisting of the first $F_n$ letters of $\mathcal{S}$ is the same as the substring consisting of all the letters of $\mathcal{S}$ between the $F_{n+1}+1$\textsuperscript{th} and the $F_{n+2}$\textsuperscript{th} positions, inclusively.  This claim follows immediately from the following: 
\begin{itemize}
    \item by Property b), the first $F_n$ letters of $\mathcal{S}$ are given by $S_{n-1}$;
    \item also by Property b), the first $F_{n+2}$ letters of $\mathcal{S}$ are given by $$S_{n+1}\ :=\ S_n:S_{n-1};\mbox{ and }$$ 
    \item by Property a), $|S_n| = F_{n+1}$.
\end{itemize}
\item[d)] Griffiths \cite{Gr3} proved a neat relation between the Zeckendorf decomposition and the golden string: if $n = F_{c_1} + F_{c_2} + \cdots + F_{c_\ell}$ is the Zeckendorf decomposition of $n$, then $S_{c_k} : S_{c_{k-1}} : \cdots :S_{c_1}$ gives the first $n$ letters of the golden string. 
\item[e)] Let $N_B(n)$ denote the number of $B$'s in the first $n$ letters of $\mathcal{S}$. We recall Griffiths' \cite[Lemma 3.3]{Gr3}, which states that 
\begin{equation}\label{e2}N_B(n)\ =\ \left\lfloor \frac{n+1}{\phi}\right\rfloor,\end{equation} 
where $\phi = (1+\sqrt{5})/2$, the golden ratio. 
\item[f)] The following claims are easily proved by induction: for $m\ge 1$,
\begin{itemize}
    \item[(f1)] The $(F_{2m+1}-1)$\textsuperscript{th} letter of $\mathcal{S}$ is $B$. 
    \item[(f2)] The $F_{2m+1}$\textsuperscript{th} letter of $\mathcal{S}$ is $A$. 
    \item[(f3)] The $(2F_{2m+1}-1)$\textsuperscript{th} letter of $\mathcal{S}$ is $B$. 
    \item[(f4)] The $2F_{2m+1}$\textsuperscript{th} letter of $\mathcal{S}$ is $A$. 
\end{itemize}
Due to Proposition \ref{kp1} below, (f3) and (f4) follow immediately from (f1) and (f2). 
\end{enumerate}
\begin{prop}\label{kp1}
For $n\ge 4$, the substring of $\mathcal{S}$ consisting of the first $F_n$ letters of $\mathcal{S}$ is the same as the substring of $\mathcal{S}$ consisting of the next $F_n$ letters. 
\end{prop}

\begin{proof}
Fix $n\ge 4$. Since for $k < \ell$, $S_k$ gives the initial letters of $S_\ell$, we can write
$S_{n-1} = S_{n-3} : L$ for some finite string $L$. We have
\begin{align*}
    S_{n+1}\ =\ S_n: S_{n-1}&\ =\ (S_{n-1}:S_{n-2}): (S_{n-3}: L)\\
    &\ =\ S_{n-1}:(S_{n-2}:S_{n-3}):L\\
    &\ =\ S_{n-1}:S_{n-1}:L.
\end{align*}
Since $S_{n+1}$ gives the initial letters of $\mathcal{S}$ and $|S_{n-1}| = F_n$, we are done. 
\end{proof}

We shall use the following notation. For a string $W$, we write $W-2$ to mean the string formed by deleting the last two letters of $W$. 

\begin{lem}\label{propS-2}
For $n\ge 1$, we have
$$S_n:S_{n+1}-2 \ =\ S_{n+1}:S_n-2.$$
\end{lem}

\begin{proof}
    We prove by induction. The equality is true for $n = 1$. Inductive hypothesis: suppose that it is true for $n = \ell\ge 1$. We show that
    $$S_{\ell+1}:S_{\ell+2} - 2\ =\ S_{\ell+2}:S_{\ell+1} - 2.$$
    We have
    \begin{align*}
        S_{\ell+2}:S_{\ell+1} - 2&\ =\ (S_{\ell+1}: S_{\ell}): S_{\ell+1} - 2 \ =\ S_{\ell+1}:(S_{\ell}: S_{\ell+1} - 2)\\
        &\ =\ S_{\ell+1}:(S_{\ell+1}:S_{\ell}-2)\ =\ S_{\ell+1}:S_{\ell+2}-2. 
    \end{align*}
\end{proof}

\begin{prop}\label{kp2}
For $n\ge 5$, the substring of $\mathcal{S}$ consisting of the first $(F_{n-1}-2)$ letters is the same as the substring of $\mathcal{S}$ consisting of all the letters between the $(2F_n+1)$\textsuperscript{th} letter and the $(F_{n+2}-2)$\textsuperscript{th} letter, inclusively.    
\end{prop}

\begin{proof}
Pick $n\ge 5$. We have 
\begin{align*}
S_{n+1}\ =\ S_{n}:S_{n-1}&\ =\ S_{n-1}:S_{n-2}:S_{n-2}:S_{n-3}\\
&\ =\ S_{n-1}:S_{n-2}:S_{n-3}: S_{n-4} : S_{n-3}\\
&\ =\ S_{n-1}:S_{n-1}:S_{n-4}:S_{n-3}. 
\end{align*}
Observe that $|S_{n-1}| = F_n$ and $|S_{n-4}:S_{n-3}| = F_{n-3}+F_{n-2} = F_{n-1}$. 
Hence, the substring of $\mathcal{S}$ consisting of all the letters between the $(2F_n+1)$\textsuperscript{th} letter and the $(F_{n+2}-2)$\textsuperscript{th} letter, inclusively is
$$W \ :=\ S_{n-4}:S_{n-3}-2.$$
By Lemma \ref{propS-2}, 
$$W \ =\ S_{n-3}:S_{n-4}-2\ =\ S_{n-2} - 2,$$
which gives the first $F_{n-1}-2$ letters of $\mathcal{S}$. 
\end{proof}

\section{The ordered list of positive integers $n$ with $F_{2k} = \min\mathcal{CG}(n)$}\label{tableprop}

For $k\ge 1$, let $A_{2k} = \{n: F_{2k} = \min\mathcal{CG}(n)\}$. We form a table whose rows are numbers in $A_{2k}$ arranged in increasing order $q(1) < q(2) < q(3) < \cdots$.  Let us look at the first few rows of the table. 

\begin{center}
\begin{tabular}{  c c c c c c }
 $q(1)$&$F_{2k}$ & & & &  \\ 
 $q(2)$&$2F_{2k}$ & & & &  \\  
 $q(3)$&$F_{2k}$ & & $F_{2k+2}$ & & \\  
 $q(4)$&$2F_{2k}$ & &  $F_{2k+2}$ & & \\  
 $q(5)$&$F_{2k}$ & &  $2F_{2k+2}$ & & \\  
$q(6)$&$F_{2k}$ &  &  & & $F_{2k+4}$\\ 
 $q(7)$&$2F_{2k}$ &  & & & $F_{2k+4}$\\  
 $q(8)$&$F_{2k}$ & & $F_{2k+2}$ & & $F_{2k+4}$\\  
 $q(9)$&$2F_{2k}$ & & $F_{2k+2}$ & & $F_{2k+4}$\\  
 $q(10)$&$F_{2k}$ & & $2F_{2k+2}$ & & $F_{2k+4}$\\  
  $q(11)$&$F_{2k}$ &  & & &  $2F_{2k+4}$\\ 
 $q(12)$&$2F_{2k}$ &  & & & $2F_{2k+4}$\\  
 $q(13)$&$F_{2k}$ & &  $F_{2k+2}$ & & $2F_{2k+4}$\\  
 \vdots
\end{tabular}

\centering{Table 1. The numbers in $A_{2k}$ in increasing order}
\end{center}

This section shows a way to construct new rows of Table 1 recursively.  
First, we prove that a number with a larger maximum Fibonacci number in its Chung-Graham decomposition belongs 
to a lower row. Consequently, the numbers $n$ with the same $\max \mathcal{CG}(n)$ form consecutive rows in the table. 

\begin{lem}\label{l1}
For $m\ge 0$, the largest positive integer $n\in A_{2k}$ with $\max\mathcal{CG}(n) = F_{2k+2m}$, denoted by $N(m)$, is 
$$F_{2k} + F_{2k+2} + \cdots + F_{2k+2m-2} + 2F_{2k+2m}.$$
\end{lem}

\begin{proof}
 Write $N(m)$ as $\sum_{i=0}^m c_i F_{2k+2i}$, where the $c_i$'s are in $\{0, 1, 2\}$ and satisfy the Chung-Graham condition. Suppose, for a contradiction, that $c_j = 0$ for some $1 \le j \le m-1$. If changing $c_j$ to $1$ does not violate the Chung-Graham condition, then $N(m) + F_{2k+2j}$ is greater than $N(m)$, while $\max\mathcal{CG}(N(m)+F_{2k+2j}) = F_{2k+2m}$. This contradicts the maximality of $N(m)$. Hence, changing $c_j$ to $1$ violates the Chung-Graham condition; that is, there are $1\le j' < j < j''\le m$ such that $c_{j'} = c_{j''} = 2$. Here we choose the largest $j'$ and the smallest $j''$ that satisfy these conditions. We change both $c_{j'}$ and $c_{j}$ to $1$. Then the new coefficients still satisfy the Chung-Graham condition, but since 
 $$F_{2k+2j'} + F_{2k+2j}\ >\ 2F_{2k+2j'} + 0F_{2k+2j},$$
 the new number is greater than $N(m)$. This again contradicts the maximality of $N(m)$. Therefore, $c_i\ge 1$ for all $i$, which clearly implies that 
 $$N(m) \ =\ F_{2k} + F_{2k+2} + \cdots + F_{2k+2m-2} + 2F_{2k+2m},$$
 as desired. 
\end{proof}

\begin{cor}\label{c1}
Given $n, m\in A_{2k}$, if $\max\mathcal{CG}(n) < \max \mathcal{CG}(m)$, then 
$n < m$. 
\end{cor}
\begin{proof}
Suppose that $\max\mathcal{CG}(n) = F_{2k+2j_1}$ and $\max\mathcal{CG}(m) = F_{2k+2j_2}$ with $j_1 < j_2$. Then $m\ge F_{2k} + F_{2k+2j_2}$. By Lemma \ref{l1}, 
\begin{align*}
n&\ \le\ F_{2k} + F_{2k+2} + \cdots + F_{2k+2j_1-2} + 2F_{2k+2j_1}\\ 
&\ =\ (F_{2k-1} + F_{2k} + F_{2k+2} + \cdots + F_{2k+2j_1-2}) + 2F_{2k+2j_1} - F_{2k-1}\\
&\ =\ F_{2k+2j_1 - 1} + 2F_{2k+2j_1} - F_{2k-1}\ =\ F_{2k+2j_1 + 2} - F_{2k-1} \ <\ F_{2k+2j_2}\ <\ m.
\end{align*}
\end{proof}

Thanks to Corollary \ref{c1}, we know that the numbers having $F_{2k+2\ell}$ ($\ell\ge 1$) (either coefficient $1$ or $2$) as the maximum Fibonacci number in its decomposition form consecutive rows in Table 1. The next result tells us their location. 

\begin{lem}\label{l2} For $\ell\ge 1$, the numbers in $\{n\in A_{2k}: \max\mathcal{CG}(n) = F_{2k + 2\ell}\}$
lie between the  $(F_{2\ell+1}+1)$\textsuperscript{th} and the $F_{2\ell+3}$\textsuperscript{th} rows, inclusively.
\end{lem}

\begin{proof} We prove by induction. Base case: it is easy to verify that the lemma holds for $\ell=1$. Inductive hypothesis: suppose the lemma is true for $\ell\le m$ for some $m\geq1$. We prove that the lemma holds for $\ell=m+1$; that is, the numbers in $\{n\in A_{2k}: \max\mathcal{CG}(n) = F_{2k + 2m+2}\}$ lie between the  $(F_{2m+3}+1)$\textsuperscript{th} to the $F_{2m+5}$\textsuperscript{th} rows, inclusively. By the inductive hypothesis,
\begin{align*}
&|\{n\in A_{2k}: \max\mathcal{CG}(n) \le F_{2k+2m}\}|\\
\ =\ &\sum_{\ell = 0}^m |\{n\in A_{2k}: \max\mathcal{CG}(n) = F_{2k+2\ell}\}|\\
\ =\ &2 + \sum_{\ell = 1}^m |\{n\in A_{2k}: \max\mathcal{CG}(n) = F_{2k+2\ell}\}|\\
\ =\ &F_3 + \sum_{\ell=1}^{m} (F_{2\ell+3} - F_{2\ell+1})\ =\ F_3 + \sum_{\ell = 1}^m F_{2\ell + 2}\ =\ F_{2m+3}. 
\end{align*}
Therefore, the numbers with the largest summand $F_{2k+2m+2}$ start at the $(F_{2m+3}+1)$\textsuperscript{th} row. 

It suffices to prove that 
$$|\{n\in A_{2k}\,:\, \max\mathcal{CG}(n) = F_{2k + 2m + 2}\}| \ =\ F_{2m+4}.$$
Note that 
\begin{align*}&\{n\in A_{2k}: \max\mathcal{CG}(n) = F_{2k + 2m+2}\}\\
\ =\ &\{n\in A_{2k}:  \max\mathcal{CG}(n) = \max\mathcal{CG}_1(n) = F_{2k+2m+2}\}\\
&\cup \{n\in A_{2k}: \max\mathcal{CG}(n) = \max\mathcal{CG}_2(n) = F_{2k+2m+2}\}.\end{align*}
Since all numbers in $\{n\in A_{2k}: \max\mathcal{CG}(n) = \max\mathcal{CG}_1(n) = F_{2k+2m+2}\}$ are created by adding $F_{2k+2m+2}$ to the numbers in $\{n\in A_{2k}: \max\mathcal{CG}(n) \le F_{2k+2m}\}$, we know that
$$|\{n\in A_{2k}: \max\mathcal{CG}(n) = \max\mathcal{CG}_1(n) = F_{2k+2m+2}\}|\ =\ F_{2m+3}.$$

It remains to show
$$|\{n\in A_{2k}: \max\mathcal{CG}(n) = \max\mathcal{CG}_2(n) = F_{2k+2m+2}\}|\ =\ F_{2m+2}.$$
All numbers in $\{n\in A_{2k}: \max\mathcal{CG}(n) = \max\mathcal{CG}_2(n) = F_{2k+2m+2}\}$ are formed by adding $2F_{2k+2m+2}$ to a subset of $\{n\in A_{2k}: \max\mathcal{CG}(n) \le F_{2k+2m}\}$. Pick $s\in \{n\in A_{2k}: \max\mathcal{CG}(n) \le F_{2k+2m}\}$ such that $s+2F_{2k+2m+2} \in \{n\in A_{2k}: \max\mathcal{CG}(n) = \max\mathcal{CG}_2(n) = F_{2k+2m+2}\}$. Equivalently, if the Chung-Graham decomposition of $s$ is $\sum_{i=0}^m c_i F_{2k+2i}$, then $(c_0, c_1, \ldots, c_m)$ must have one of the following forms, based on the largest $i$ (if any) with $c_i = 0$:
\begin{align*}
f_1&\quad  (c_0, c_1, \ldots, c_{m-3}, c_{m-2}, c_{m-1}, 0)\\
f_2&\quad (c_0, c_1, \ldots, c_{m-3}, c_{m-2}, 0, 1)\\
f_3&\quad (c_0, c_1, \ldots, c_{m-3}, 0, 1, 1)\\
\quad\quad\vdots\\
f_m&\quad (c_0, 0, \ldots, 1, 1, 1, 1)\\
f_{m+1}&\quad (1, 1, \ldots, 1, 1, 1, 1).
\end{align*}
For $1\le i\le m-1$, the number of $s$ having the form $f_i$ is equal to $|\{n\in A_{2k}: \max\mathcal{CG}(n) \le F_{2k+2m-2i}\}|$, which, by the inductive hypothesis, is 
$$\sum_{j=0}^{m-i} |\{n\in A_{2k}: \max\mathcal{CG}(n) = F_{2k+2j}\}|\ =\ 2 + \sum_{j=1}^{m-i} F_{2j+2}\ =\ F_{2m-2i+3}.$$
The number of $s$ having the form $f_m$ and $f_{m+1}$ is $2$ and $1$, respectively. 
Therefore, 
\begin{align*}
    |\{n\in A_{2k}: \max \mathcal{CG}(n) = \max\mathcal{CG}_2(n) = F_{2k+2m+2}\}|&\ =\ \sum_{i=1}^{m-1} F_{2m-2i+3} + 2 + 1\\
    &\ =\ F_4 + F_5 + \cdots + F_{2m+1}\\
    &\ =\  F_{2m+2}.
\end{align*}
This completes our proof. 
\end{proof}

The proof of Lemma \ref{l2} also reveals the counts of numbers having $F_{2k+2\ell}$ and $2F_{2k+2\ell}$ as the maximum summand in their decomposition. 

\begin{cor}\label{c2}
For $\ell\ge 1$, we have
\begin{align*}
&|\{n\in A_{2k}: \max\mathcal{CG}(n) = \max\mathcal{CG}_1(n) = F_{2k + 2\ell}\}| \ =\ F_{2\ell+1},\mbox{ and }\\
&|\{n\in A_{2k}: \max\mathcal{CG}(n) = \max\mathcal{CG}_2(n) = F_{2k + 2\ell}\}| \ =\ F_{2\ell}. 
\end{align*} 
\end{cor}

We now show that the numbers in $\{n\in A_{2k}: \max\mathcal{CG}(n) = \max\mathcal{CG}_2(n) = F_{2k + 2\ell}\}$ must lie below (are larger than) the numbers in $\{n\in A_{2k}: \max\mathcal{CG}(n) = \max\mathcal{CG}_1(n) = F_{2k + 2\ell}\}$.

\begin{lem}\label{l3}For $\ell\ge 0$, 
\begin{align*}&\min \{n\in A_{2k}: \max\mathcal{CG}(n) = \max\mathcal{CG}_2(n) = F_{2k + 2\ell}\}\\
\ >\ &\max\{n\in A_{2k}: \max\mathcal{CG}(n) = \max\mathcal{CG}_1(n) = F_{2k + 2\ell}\}.\end{align*}
\end{lem}

\begin{proof}
When $\ell = 0$, we are comparing $F_{2k}$ and $2F_{2k}$, and the lemma is obviously true. Suppose that $\ell \ge 1$. Since $\{n\in A_{2k}: \max\mathcal{CG}(n) = \max\mathcal{CG}_1(n) = F_{2k + 2\ell}\}$ is formed by adding $F_{2k+2\ell}$ to each number in $\{n\in A_{2k}: \max\mathcal{CG}(n) \le F_{2k + 2\ell-2}\}$, Lemma \ref{l1} gives
\begin{align*}\max\{n\in A_{2k}&: \max\mathcal{CG}(n) = \max\mathcal{CG}_1(n) = F_{2k + 2\ell}\}\\
&\ =\ F_{2k} + F_{2k+2} + \cdots + F_{2k+2\ell-4} + 2F_{2k + 2\ell-2} + F_{2k+2\ell}.\end{align*}
On the other hand, 
$$\min \{n\in A_{2k}: \max\mathcal{CG}(n) = \max\mathcal{CG}_2(n) = F_{2k + 2\ell}\} \ =\ F_{2k} + 2F_{2k + 2\ell}.$$
We need only to verify that
$$F_{2k} + 2F_{2k + 2\ell}\ >\ F_{2k} + F_{2k+2} + \cdots + F_{2k+2\ell-4} + 2F_{2k + 2\ell-2} + F_{2k+2\ell}.$$
Equivalently,
$$\label{e1}F_{2k+1} + F_{2k+2\ell}\ >\ F_{2k+1} + F_{2k+2} + \cdots + F_{2k+2\ell-4} + 2F_{2k + 2\ell-2},$$
which is true because the right side of the inequality is equal to $F_{2k+2\ell}$. 
\end{proof}

We have used in the proof of Lemmas \ref{l2} and \ref{l3} the fact that for $\ell\ge 1$, 
\begin{align*}&\{n\in A_{2k}\,:\, \max\mathcal{CG}(n) = \max\mathcal{CG}_1(n) = F_{2k + 2\ell}\}\\
\ =\ &\{n\in A_{2k}\,:\, \max\mathcal{CG}(n) \le F_{2k + 2\ell - 2}\} + F_{2k+2\ell}.
\end{align*}
In other words, integers in $\{n\in A_{2k}\,:\, \max\mathcal{CG}(n) = \max\mathcal{CG}_1(n) = F_{2k + 2\ell}\}$ are formed by adding $F_{2k+2\ell}$ to the all previous rows in Table 1. We now describe how to form integers in $\{n\in A_{2k}\,:\, \max\mathcal{CG}(n) = \max\mathcal{CG}_2(n) = F_{2k + 2\ell}\}$. 

\begin{lem}\label{l5} For $\ell\ge 1$, we have
\begin{align*}&\{n\in A_{2k}\,:\, \max\mathcal{CG}(n) = \max\mathcal{CG}_2(n) = F_{2k + 2\ell}\}\\
\ =\ &\left\{n\in A_{2k}\,:\, n \le \sum_{i=0}^{\ell-1}F_{2k+2i}\right\} + 2F_{2k+2\ell}.
\end{align*}
\end{lem}

\begin{proof}
By Lemma \ref{l1}, the largest integer $n\in A_{2k}$ with $\max\mathcal{CG}(n) = \max\mathcal{CG}_2(n) = F_{2k + 2\ell}$ is
$$\sum_{i=0}^{\ell-1}F_{2k+2i} + 2F_{2k+2\ell};$$
hence, \begin{align*}&\{n\in A_{2k}\,:\, \max\mathcal{CG}(n) = \max\mathcal{CG}_2(n) = F_{2k + 2\ell}\}\\
\ \subset\ &\left\{n\in A_{2k}\,:\, n \le \sum_{i=0}^{\ell-1}F_{2k+2i}\right\} + 2F_{2k+2\ell}.
\end{align*}
We prove the reverse inclusion. Pick $n\in A_{2k}$ such that 
\begin{equation}\label{e4} n \ \le\ \sum_{i=0}^{\ell-1}F_{2k+2i}.\end{equation}
Write the Chung-Graham decomposition of $n$ as
$$n\ =\ c_0F_{2k} + c_1F_{2k+2} + \cdots + c_{\ell-1} F_{2k+2\ell-2},$$
for $c_i\in \{0, 1, 2\}$. 
Suppose, for a contradiction, that 
$$c_0F_{2k} + c_1F_{2k+2} + \cdots + c_{\ell-1} F_{2k+2\ell-2} + 2F_{2k+2\ell}$$ 
is not a Chung-Graham decomposition. Then there is $0\le j\le \ell-1$ such that $c_j = 2$ and $c_{i} = 1$ for all $i\in [j+1, \ell-1]$. It follows that
\begin{equation}\label{e5}n \ \ge\ 2F_{2k+2j} + \sum_{i=j+1}^{\ell-1} F_{2k+2i}\end{equation}
From \eqref{e4} and \eqref{e5}, 
$$\sum_{i=0}^{\ell-1}F_{2k+2i}\ \ge\ 2F_{2k+2j} + \sum_{i=j+1}^{\ell-1} F_{2k+2i}.$$
Equivalently, 
$$\sum_{i=0}^{j} F_{2k+2i}\ \ge\ 2F_{2k+2j},$$
which is a contradiction. Hence,
$$c_0F_{2k} + c_1F_{2k+2} + \cdots + c_{\ell-1} F_{2k+2\ell-2} + 2F_{2k+2\ell}$$
is a Chung-Graham decomposition. Therefore, 
$$n+2F_{2k+2\ell}\in \{n\in A_{2k}\,:\, \max\mathcal{CG}(n) = \max\mathcal{CG}_2(n) = F_{2k + 2\ell}\}.$$
This completes our proof. 
\end{proof}

\section{Proof of the main theorem}\label{maintheoproof}

Using what we know about Table 1 from Section \ref{tableprop}, we now prove an identity that equates the difference between two earlier consecutive rows with the difference between two later consecutive rows in the table.

\begin{prop}\label{kp} Fix $\ell\ge 1$. 
For $1+F_{2\ell+1}\le j \le F_{2\ell+3}-1$, we have
\begin{equation}\label{e27}q(j+1)-q(j)\ =\ q(j-F_{2\ell+1}+1)-q(j-F_{2\ell+1}).\end{equation}
\end{prop}

\begin{proof}
Choose $\ell\ge 1$. By Lemma \ref{l2}, the numbers in $\{n: \max\mathcal{CG}(n)\le F_{2k+2\ell-2}\}$ lie from the $1$\textsuperscript{st} row to the $F_{2\ell+1}$\textsuperscript{th} row, inclusively. Since
\begin{align*}\{n: \max\mathcal{CG}(n)\le F_{2k+2\ell-2}\} &+ F_{2k+2\ell}\\
&\ =\ \{n: \max\mathcal{CG}(n) = \max\mathcal{CG}_1(n) = F_{2k+2\ell}\},\end{align*}
we know that 
\begin{equation}\label{e20}q(i+1)-q(i)\ =\ q(i+F_{2\ell+1}+1)-q(i+F_{2\ell+1})\end{equation}
whenever $1\le i\le F_{2\ell+1}-1$.
Applying the change of variable $j=i+F_{2\ell+1}$, we obtain 
\begin{equation}\label{e25}q(j+1)-q(j) \ =\ q(j-F_{2\ell+1}+1)-q(j-F_{2\ell+1}), 1+F_{2\ell+1}\le j\le 2F_{2\ell+1}-1.\end{equation}

Furthermore, by Corollary \ref{c2} and Lemmas \ref{l3} and \ref{l5}, 
\begin{equation}\label{e21}q(i+1) - q(i) \ =\ q(i+2F_{2\ell+1}+1)-q(i+2F_{2\ell+1}), 1\le i\le F_{2\ell}-1.\end{equation}
From \eqref{e20} and \eqref{e21}, 
$$q(i+F_{2\ell+1}+1)-q(i+F_{2\ell+1})\ =\ q(i+2F_{2\ell+1}+1)-q(i+2F_{2\ell+1}), 1\le i\le F_{2\ell}-1.$$
Using the change of variable $j=i+2F_{2\ell+1}$, we have 
\begin{equation}\label{e26}q(j+1)-q(j)=q(j-F_{2\ell+1}+1)-q(j-F_{2\ell+1}), 1+2F_{2\ell+1}\le j\le F_{2\ell+3} - 1.\end{equation}

Thanks to \eqref{e25} and \eqref{e26}, it remains to verify \eqref{e27} when $j = 2F_{2\ell+1}$; that is,
\begin{equation}\label{e28}q(2F_{2\ell+1}+1)-q(2F_{2\ell+1})\ =\ q(F_{2\ell+1}+1)-q(F_{2\ell+1}).\end{equation}
By Lemmas \ref{l1}, \ref{l2},\ref{l3}, and Corollary \ref{c2},  
\begin{align*}
    q(F_{2\ell+1}) &\ =\ \sum_{i=0}^{\ell-2} F_{2k+2i} + 2F_{2k+2\ell-2},\\
    q(F_{2\ell+1} + 1) &\ =\ F_{2k} + F_{2k+2\ell},\\
    q(2F_{2\ell+1})&\ =\ F_{2k} + F_{2k+2} + \cdots + F_{2k+2\ell-4} + 2F_{2k + 2\ell-2} + F_{2k+2\ell},\\
    q(2F_{2\ell+1}+1)&\ =\ F_{2k} + 2F_{2k+2\ell}.
\end{align*}
These confirm \eqref{e28}, and we are done.
\end{proof}

We are now ready to prove the following key lemma to describe all integers in $A_{2k}$.
\begin{lem}\label{kl}
For $j\ge 2$,
\begin{equation}\label{e30}
   q(j+1)-q(j)\ =\ \begin{cases}F_{2k}\mbox{ if the }(j-1)\mbox{\textsuperscript{th} letter of }\mathcal{S}\mbox{ is } A,\\ F_{2k+1}\mbox{ if the }(j-1)\mbox{\textsuperscript{th} letter of }\mathcal{S}\mbox{ is } B.\end{cases} 
\end{equation}
\end{lem}

\begin{proof}
It suffices to prove that \eqref{e30} is true for all $j\le F_{2m+1}-1$ for any arbitrary $m\in \mathbb{N}$. We do so by induction.

Base case: for $m=3$, we can see from Table 1 that \eqref{e30} is true for all $j\le F_{7}-1$. 

Inductive hypothesis: suppose that \eqref{e30} is true for $j\le F_{2m+1}-1$ for some $m\ge 3$. We need to show that it is true for all $j\le F_{2m+3}-1$. We proceed by a case analysis.

Case 1: $j = F_{2m+1}$. By Lemmas \ref{l1} and \ref{l2}, we have
\begin{align*}
q(F_{2m+1}+1)-q(F_{2m+1})&\ =\ F_{2k} + F_{2k+2m} - (F_{2k} + \cdots + F_{2k+2m-4} + 2F_{2k+2m-2})\\
&\ =\ F_{2k+1}.
\end{align*}
By Property f) in Section \ref{gold}, \eqref{e30} is true when $j = F_{2m+1}$. 

Case 2: $j = F_{2m+1}+1$. By Lemma \ref{l2}, 
$$
q(F_{2m+1}+2)-q(F_{2m+1}+1)\ = \ (2F_{2k} + F_{2k+2m}) - (F_{2k}+F_{2k+2m})\ =\ F_{2k}. 
$$
By Property f) in Section \ref{gold}, \eqref{e30} is true when $j = F_{2m+1}+1$.

Case 3: $F_{2m+1}+2\le j\le 2F_{2m+1}-1$. It follows from Proposition \ref{kp} that 
$$q(j+1)-q(j)\ =\ q(j+1-F_{2m+1}) - q(j-F_{2m+1}).$$
Thanks to Proposition \ref{kp1} and the fact that
$$j - F_{2m+1}\ \le\ 2F_{2m+1}-1-F_{2m+1} \ =\ F_{2m+1}-1,$$
the inductive hypothesis guarantees that \eqref{e30} is true for $F_{2m+1}+2\le j\le 2F_{2m+1}-1$.

Case 4: $j = 2F_{2m+1}$. It follows from Lemmas \ref{l1}, \ref{l2}, \ref{l3}, and Corollary \ref{c2} that 
\begin{align*}
&q(2F_{2m+1}+1) - q(2F_{2m+1})\\
\ =\ &(F_{2k} + 2F_{2k+2m})-(F_{2k} + \cdots + 2F_{2k+2m-2} + F_{2k+2m})\ =\ F_{2k+1}. 
\end{align*}
Property f) in Section \ref{gold} confirms that \eqref{e30} is true for $j = 2F_{2m+1}$. 

Case 5: $j = 2F_{2m+1}+1$. By Lemmas \ref{l2}, \ref{l3}, and Corollary \ref{c2}, 
$$q(2F_{2m+1}+2) - q(2F_{2m+1}+1)\ =\ (2F_{2k} + 2F_{2k+2m}) - (F_{2k} + 2F_{2k+2m})\ =\ F_{2k}.$$
Property f) in Section \ref{gold} confirms that \eqref{e30} is true for $j = 2F_{2m+1}+1$.

Case 6: $2F_{2m+1}+2\le j\le F_{2m+3}-1$. According to Proposition \ref{kp}, 
$$q(j+1)-q(j)\ =\ q(j+1-2F_{2m+1}) - q(j-2F_{2m+1}).$$
Since 
$$j-2F_{2m+1}\ \le\ F_{2m+3}-1-2F_{2m+1}\ =\ F_{2m}-1,$$
the inductive hypothesis can be applied. Together with Proposition \ref{kp2}, we know that \eqref{e30} holds for 
$2F_{2m+1}+2\le j\le F_{2m+3}-1$.

This completes our proof. 
\end{proof}

\begin{prop} For $k\ge 1$, we have
$$A_{2k} \ =\ \left\{F_{2k},  (n+2)F_{2k} + \left\lfloor \frac{n+1}{\phi}\right\rfloor F_{2k-1}\,:\,n\ge 0\right\}.$$
\end{prop}

\begin{proof}
By Lemma \ref{kl}, we have
$$A_{2k} \ =\ \{F_{2k},  2F_{2k} + a(n)F_{2k} + b(n)F_{2k+1}\,:\,n\ge 0\},$$
where $a(n)$ and $b(n)$ are the number of $A$'s and $B$'s, respectively, among the first $n$ letters of $\mathcal{S}$. Due to \eqref{e2}, 
\begin{align*}
A_{2k}&\ =\ \left\{F_{2k},  2F_{2k} + \left(n-\left\lfloor \frac{n+1}{\phi}\right\rfloor\right)F_{2k} + \left\lfloor \frac{n+1}{\phi}\right\rfloor F_{2k+1}\,:\,n\ge 0\right\}\\
&\ =\ \left\{F_{2k},  (n+2)F_{2k} + \left\lfloor \frac{n+1}{\phi}\right\rfloor F_{2k-1}\,:\,n\ge 0\right\}.
\end{align*}
\end{proof}

\begin{proof}[Proof of Theorem \ref{mtheo}]
For $1\le N < k$, the set of integers whose Chung-Graham decomposition has $F_{2k}$ but none of $F_{2N}, \ldots, F_{2k-2}$ is 
$$\left\{j+F_{2k},  j+(n+2)F_{2k} + \left\lfloor \frac{n+1}{\phi}\right\rfloor F_{2k-1}\,:\, 0\le j\le F_{2N}-1, n\ge 0\right\}.$$ 
Indeed,  all Chung-Graham decompositions of the form $\sum_{i=1}^{N-1} c_i F_{2i}$, when added to an integer in $A_{2k}$, gives an integer whose Chung-Graham decomposition has $F_{2k}$ but none of $F_{2N}, \ldots, F_{2k-2}$. Meanwhile, 
$$\left\{\sum_{i=1}^{N-1} c_i F_{2i}\,:\, c_i\mbox{'s satisfy the Chung-Graham decomposition}\right\}\ =\ [0, F_{2N}-1].$$
Therefore,
\begin{align*} B_{2N}&\ :=\ [1, F_{2N}-1]\cup\\
&\bigcup_{k=N+1}^\infty \left\{j + F_{2k}, j+(n+2)F_{2k} + \left\lfloor \frac{n+1}{\phi}\right\rfloor F_{2k-1}\,:\, 0\le j\le F_{2N}-1, n\ge 0\right\}.\end{align*}
\end{proof}


\ \\
\end{document}